\documentclass[11pt]{article}

\usepackage{geometry}
\usepackage{cite}
\usepackage{graphicx}
\usepackage{bm}
\usepackage{booktabs} 
\usepackage{array} 
\usepackage{paralist} 
\usepackage{verbatim} 
\usepackage{subfig} 
\usepackage{tabularx}
\usepackage{fancyhdr} 
\usepackage[nottoc,notlof,notlot]{tocbibind}
\usepackage[titles,subfigure]{tocloft}

\usepackage{amsmath,amsfonts,amsthm,mathrsfs,amssymb,cite}
\usepackage[usenames]{color}

\def\BC{\mathbb C}
\def\BN{\mathbb N}
\def\BR{\mathbb R}

\def\cD{\mathcal D}

\def\rd{\mathrm d}
\def\rRe{\mathrm{Re}}

\def\e{\mathrm e}

\def\ri{\mathrm i}

\def\supp{\mathrm{supp}}

\def\Ga{\Gamma}

\def\Om{\Omega}
\def\al{\alpha}
\def\be{\beta}
\def\ga{\gamma}
\def\de{\delta}
\def\ep{\epsilon}
\def\ve{\varepsilon}

\def\la{\lambda}
\def\si{\sigma}
\def\vp{\varphi}
\def\om{\omega}
\def\f{\frac}
\def\nb{\nabla}
\def\ov{\overline}
\def\pa{\partial}

\def\wt{\widetilde}
\def\tri{\triangle}

\newtheoremstyle{thmstyle}
  {6pt}
  {6pt}
  {\it}
  {}
  {\bf}
  {}
  {.5em}
  {}
\newtheoremstyle{remstyle}
  {6pt}
  {6pt}
  {\rm}
  {}
  {\bf}
  {}
  {.5em}
  {}

\theoremstyle{thmstyle}
\newtheorem{thm}{\indent Theorem}[section]
\newtheorem{lem}[thm]{\indent Lemma}

\newtheorem{prob}[thm]{\indent Problem}
\theoremstyle{remstyle}
\newtheorem{rem}[thm]{\indent Remark}

\allowdisplaybreaks[4]

\title{\bf Inverse moving source problem for 
time-fractional evolution equations: Determination of profiles}

\author{Yikan Liu\thanks{Research Center of Mathematics for Social Creativity, Research Institute for Electronic Science, Hokkaido University, N12W7, Kita-Ward, Sapporo 060-0812, Japan. Email: {\tt ykliu@es.hokudai.ac.jp}},
Guanghui Hu\thanks{School of Mathematical Sciences and LPMC, Nankai University, Tianjin 300071, P.R. China. Email: {\tt ghhu@nankai.edu.cn}},
Masahiro Yamamoto\thanks{Graduate School of Mathematical Sciences, The University of Tokyo, 3-8-1 Komaba, Meguro-ku, Tokyo 153-8914, Japan; Honorary Member of Academy of Romanian Scientists, Ilfov, nr.\! 3, Bucuresti, Romania; Correspondence member of Accademia Peloritana dei Pericolanti, Palazzo Universit\`a, Piazza S. Pugliatti 1 98122 Messina, Italy; Peoples' Friendship University of Russia (RUDN University), 6 Miklukho-Maklaya Street, Moscow 117198, Russian Federation. Email: {\tt myama@ms.u-tokyo.ac.jp}}}

\date{}

\begin{document}
\maketitle

\begin{abstract}
This article is concerned with two inverse problems on determining moving source profile functions in evolution equations with a derivative order $\al\in(0,2]$ in time. In the first problem, the sources are supposed to move along known straight lines, and we suitably choose partial interior observation data in finite time. Reducing the problems to the determination of initial values, we prove the unique determination of one and two moving source profiles for $0<\al\le1$ and $1<\al\le2$, respectively. In the second problem, the orbits of moving sources are assumed to be known, and we consider the full lateral Cauchy data. At the cost of infinite observation time, we prove the unique determination of one moving source profile by constructing test functions.
\vskip 4.5mm
\noindent\begin{tabular}{@{}l@{ }p{10cm}} {\bf Keywords } & inverse moving source problem, time-fractional evolution equation, vanishing property, uniqueness
\end{tabular}

\vskip 4.5mm

\end{abstract}
\section{Introduction}

Let $0<\al\le2$, $T>0$ and $\Om\subset\BR^d$ ($d\ge2$) be a bounded domain with a smooth boundary $\pa\Om$. We may consider $T=\infty$ in some cases. Consider an initial-boundary value problem for a time-fractional evolution equation
\begin{equation}\label{eq-IBVP-u}
\begin{cases}
(\pa_{0+}^\al-\tri)u=F & \mbox{in }\Om\times(0,T),\\
u=\pa_t^{\lceil\al\rceil-1}u=0 & \mbox{in }\Om\times\{0\},\\
u=0 & \mbox{on }\pa\Om\times(0,T),
\end{cases}
\end{equation}
which may be a parabolic or a hyperbolic equation, that is, $\al=1$ or $\al=2$. Here $\tri:=\sum_{j=1}^d\f{\pa^2}{\pa x_j^2}$ denotes the usual Laplacian with respect to $\bm x$ and 
\[
\lceil\al\rceil=\begin{cases}
n & \mbox{if $n-1<\al< n$ where $n\in\BN$,}\\
\al & \mbox{if $\al\in\BN$}.
\end{cases}
\]
The notation $\pa_{0+}^\al$ stands for the forward Caputo derivative in time, which will be defined precisely in Section \ref{sec-premain}. The equation \eqref{eq-IBVP-u} is called a time-fractional diffusion equation when $\al\in(0,1)$, whereas it is called a time-fractional wave equation when $\al\in(1,2)$. In practice, the Caputo derivative has been extensively used to describe non-Fickian dispersion, skewness and long-tailed profiles which are poorly modeled by integer derivatives (see e.g.\! \cite{BWM00,LB03}). Correspondingly, the fractional evolution equation \eqref{eq-IBVP-u} has become a powerful candidate in modeling e.g.\! anomalous diffusion in heterogeneous media and complex viscoelastic materials (see \cite{AG92,BDES18,GCR92,HH98} and the references therein).

In the first part of this paper, the source term $F$ in \eqref{eq-IBVP-u} is assumed to take the form
\begin{equation}\label{eq-def-F}
F(\bm x,t):=\begin{cases}
f(\bm x-\bm p t), & 0<\al\le1,\\
f(\bm x-\bm p t)+g(\bm x-\bm q t), & 1<\al\le2,
\end{cases}
\end{equation}
where $\bm p,\bm q\in\BR^d$ are constant vectors. Throughout this paper, we assume
\[
f,g\mbox{ are compactly supported in }B_{\de_0}:=\{\bm x\in\BR^d;\,|\bm x|<\de_0\}\subset\Om
\]
for some $\de_0>0$, whose regularity will be specified later in Section \ref{sec-premain}. Then for $0<\al\le1$, the function $F$ describes a radiating source which moves along the direction $\bm p$ with the source profile function $f$ and the velocity $\bm p$. For $1<\al\le2$, the function $F$ models two radiating sources moving along the directions $\bm p,\bm q$ whose profiles and velocities are $f,g$ and $\bm p,\bm q$, respectively.  In other words, the function $F(\,\cdot\,,t)$ is compactly supported in $\Om$ for any $t\in(0,T)$, which implies that the moving sources under consideration will not move beyond $\Om$ for any $t\in(0,T)$. Note that in the case $\al=2$, the problem \eqref{eq-IBVP-u} models the acoustic wave propagation in a homogeneous isotropic background medium with the unit wave velocity.

We first consider the following inverse problem on determining one or two moving source profiles.

\begin{prob}\label{prob-IMSP}
Let $u$ be the solution to \eqref{eq-IBVP-u} with \eqref{eq-def-F}, and $\om\subset\Om$ be a suitably chosen nonempty subdomain of $\Om$. Provided that $\bm p,\bm q\in\BR^d$ are known constant vectors such that $\bm p\ne\bm q$, determine one source profile $f$ in the case of $0<\al\le1$ or two source profiles $f,g$ in the case of $1<\al\le2$ in \eqref{eq-def-F} by the partial interior observation of $u$ in $\om\times(0,T)$.
\end{prob}

Problem \ref{prob-IMSP} with $1<\al\le2$ requires the simultaneous determination of $f$ and $g$, which definitely includes the case of determining a single source profile. In this paper, we are concerned with the uniqueness issue of Problem \ref{prob-IMSP}. Due to the linearity of the problem, we assume additionally
\begin{equation}\label{eq-u-0}
u=0\quad\mbox{in }\om\times(0,T).
\end{equation}
Then it suffices to verify $f=g\equiv0$ in $B_{\de_0}$.

In the second part of this paper, the source term $F$ in \eqref{eq-IBVP-u} is assumed to take the form
\begin{equation}\label{eq:15}
F(\bm x,t)=f(\bm x-\bm\rho(t))h(t).
\end{equation}
Here, the function $\bm\rho:[0,\infty)\longrightarrow B_R\subset\Om$ denotes a smooth orbit and $f: \Om\longrightarrow\BR$ the source profile. We assume that $f\in C_0^\infty(\Om)$ is compactly supported over a ball $B_{\de_0}$ for some $\de_0>0$, so that its zero extension in $\BR^d$, still denoted by $f$, belongs to $C_0^\infty(\BR^d)$. Moreover, the temporal function $h$ is assumed to be smooth, non-vanishing and compactly supported on $[0, T_0]$ for some $T_0>0$. Suppose that $B_{R+\de_0}\subset \Om$. Since the source term cannot enter into the exterior of $B_R$, the function $F(\bm x,t)$ is compactly supported in $B_{R+\de_0}\times[0,T_0]$.

\begin{prob}\label{prob-IMSP-2}
Let $u$ be the solution to \eqref{eq-IBVP-u} with \eqref{eq:15} and assume that $\bm\rho$ and $h$ are both known. Determine the moving source profile $f$ in the case of $0<\al\le2$ in \eqref{eq:15} by the full lateral Cauchy data $(u,\pa_{\bm\nu}u)|_{\pa\Om\times(0,\infty)}$.
\end{prob}

In comparison with the first inverse problem, the orbit function $\bm\rho$ appearing in Problem \ref{prob-IMSP-2} is not restricted to the class of straight lines. However, the dynamical Cauchy data for all $t\in[0,\infty)$ is needed, because our argument relies heavily on the Laplace transform of the model and data. We shall prove the unique determination of the compactly supported source profile $f$.

In the past two decades, time-fractional evolution equations represented by \eqref{eq-IBVP-u} with $\al\in(0,1)\cup(1,2)$ have gathered increasing popularity among researchers from multiple disciplines owing to their outstanding flexibility in modeling various nonlocal phenomena. Mathematically, a series of fundamental and important results about time-fractional evolution equations, featured by well-posedness, asymptotic behavior, time-analyticity etc.\! of solutions, have been established in recent years; see \cite{EK04,L09,SY11a,GLY15,LRY16,KY18} as a partial list. Along with the completeness of theories for forward problems, inverse problems for time-fractional evolution equations have also been studied intensively from both theoretical and numerical aspects, and we refer to the review articles \cite{LiuLiY19,LiLiuY19,LiY19} as well as the references therein. Here we do not intend any comprehensive list of references.

Due to the practical significance, a lot of works have been devoted to inverse source problems for time-fractional diffusion equations, among which the majority assume that the inhomogeneous term takes the form of (partial) separated variables. We refer e.g.\! to \cite{LRY16,FK16,LZ17} and \cite{SY11b,JLLY17} for the determination of temporal and spatial components, respectively. It reveals that the treatments for the above inverse problems relies heavily on some properties of forward problems, and mostly it is technically difficult to obtain stability results because of the non-locality of time-fractional derivatives.

As a special branch of inverse source problems, there are several papers on inverse moving source problems, most of which are concerned with determining moving orbits in hyperbolic equations. In \cite{NIO12,O20}, algebraic procedures were applied to identify moving point or dipole sources. In \cite{HKLZ2019}, the authors considered inverse problems arising from the Maxwell system (that is, $\al=2$) for recovering moving source profile (respectively orbit) from boundary surface data, if a priori information on the source orbit (respectively profile) is available. In our previous work \cite{HLY20}, stability and uniqueness for fractional diffusion(-wave) equations ($0<\al\le2$) in determining the moving orbit were derived using observation data at multiple interior points, provided that the moving source profile is given. Unlike the above mentioned problems, in this paper we deal with the moving source taking the form of \eqref{eq-def-F}, where the a priori information of $\bm p$ and $\bm q$ means that we know moving orbits (directions) of the sources. The aim of this paper is to identify one or two unknown moving source profiles which do not change in the time variable. To the best of our knowledge, there seems no literature in this respect for time-fractional evolution equations.

The remaining part of this paper is organized as follows. In Section \ref{sec-premain}, we first fix notations and terminologies for fractional equations, and then state well-posedness and regularity results (Lemma \ref{lem-IBVP-u}) of the forward problem \eqref{eq-IBVP-u}--\eqref{eq-def-F} together with a unique continuation property for fractional equations with $1<\al<2$ (Lemma \ref{lem-UCP}). Our main uniqueness results for Problems \ref{prob-IMSP} and \ref{prob-IMSP-2} will be presented in Theorems \ref{thm-unique} and \ref{TH:4}, respectively. Section \ref{sec-forward} is devoted to the proofs of Lemmas \ref{lem-IBVP-u}--\ref{lem-asymp}. The proofs of Theorems \ref{thm-unique} and \ref{TH:4} will be carried out in Sections \ref{sec-proof} and \ref{sec:5}, respectively. Finally, some concluding remarks are given in Section \ref{sec-conclusion}.

\section{Preliminaries and main results}\label{sec-premain}

To start with, we recall the Riemann-Liouville integral operator for $\be\in[0,1]$:
\[
J_{0+}^\be h(t):=\left\{\!\begin{alignedat}{2}
& h(t), & \quad & \be=0,\\
& \f1{\Ga(\be)}\int_0^t\f{h(\tau)}{(t-\tau)^{1-\be}}\,\rd\tau, & \quad & 0<\be\le1,
\end{alignedat}\right.\quad h\in C[0,\infty),
\]
where $\Ga(\,\cdot\,)$ is the Gamma function. Then for $\be>0$, the Caputo derivative $\pa_{0+}^\be$ and the Riemann-Liouville derivative $D_{0+}^\be$ can be formally defined as
\[
\pa_{0+}^\be=J_{0+}^{\lceil\be\rceil-\be}\circ\f{\rd^{\lceil\be\rceil}}{\rd t^{\lceil\be\rceil}},\quad D_{0+}^\al=\f{\rd^{\lceil\be\rceil}}{\rd t^{\lceil\be\rceil}}\circ J_{0+}^{\lceil\be\rceil-\be},
\]
where $\circ$ denotes the composition. Then by direct calculations, we know

\begin{lem}\label{lem-C-RL}
Let $h\in C^\infty[0,\infty)$. Then

{\rm(a)}\ \ For $0<\al<1$, we have $\pa_{0+}^\al h=D_{0+}^\al h$ if $h(0)=0$.

{\rm(b)}\ \ For $1<\al<2$, we have $D_{0+}^\al h=\pa_t J_{0+}^{2-\al}\pa_t h$ if $h(0)=0,$ and $\pa_t J_{0+}^{2-\al}\pa_t h=\pa_{0+}^\al h$ if $h'(0)=0$.
\end{lem}

For later use, we also introduce the backward Riemann-Liouville integral operator for $T>0$ as
\[
J_{T-}^\be h(t):=\left\{\!\begin{alignedat}{2}
& h(t), & \quad & \be=0,\\
& \f1{\Ga(\be)}\int_t^T\f{h(\tau)}{(t-\tau)^{1-\be}}\,\rd\tau, & \quad & 0<\be\le1,
\end{alignedat}\right.\quad h\in C[0,\infty),
\]
by which we further define the corresponding backward Caputo and Riemann-Liouville derivatives with $\be>0$ as
\[
\pa_{T-}^\be=J_{T-}^{\lceil\be\rceil-\be}\circ\f{\rd^{\lceil\be\rceil}}{\rd t^{\lceil\be\rceil}},\quad D_{T-}^\be=\f{\rd^{\lceil\be\rceil}}{\rd t^{\lceil\be\rceil}}\circ J_{T-}^{\lceil\be\rceil-\be}.
\]
In the next lemma, we collect useful formulae connecting forward and backward fractional derivatives from \cite[Lemma 2.1]{L21}.

\begin{lem}\label{lem-frac-int}
Let $h_1,h_2\in C^{\lceil\al\rceil}[0,T]$. If $0<\al\le1,$ then
\begin{equation}\label{eq-frac-int1}
\int_0^T(\pa_{0+}^\al h_1)\,h_2\,\rd t=\left[h_1(J_{T-}^{1-\al}h_2)\right]_0^T-\int_0^T h_1\,(D_{T-}^\al h_2)\,\rd t.
\end{equation}
If $1<\al\le2,$ then
\begin{equation}\label{eq-frac-int2}
\begin{aligned}
& \int_0^T(\pa_{0+}^\al h_1)\,h_2\,\rd t=\left[h_1'(J_{T-}^{2-\al}h_2)\right]_0^T-\int_0^T h_1'\,(D_{T-}^{\al-1}h_2)\,\rd t,\\
& \int_0^T h_1'\,(D_{T-}^{\al-1}h_2)\,\rd t=\left[h_1(D_{T-}^{\al-1}h_2)\right]_0^T-\int_0^T h_1\,(D_{T-}^\al h_2)\,\rd t.
\end{aligned}
\end{equation}
\end{lem}

For the solution expression, we invoke the Mittag-Leffler function
\[
E_{\al,\be}(z):=\sum_{k=0}^\infty\f{z^k}{\Ga(\al k+\be)},\quad z\in\BC,\ \al>0,\ \be\in\BR,
\]
which satisfies the frequently used estimate (e.g., Podlubny \cite[Theorem 1.5]{P99}):
\begin{equation}\label{eq-est-ML}
|E_{\al,\be}(-\eta)|\le\f C{1+\eta},\quad\eta\ge0,\ 0<\al<2,\ \be>0.
\end{equation}

Let $L^2(\Om)$ denote the usual $L^2$-space in $\Om$ equipped with the inner product $(\,\cdot\,,\,\cdot\,)$, and let $H_0^1(\Om)$, $H^2(\Om)$ etc.\! be the standard $L^2$-based Sobolev spaces (e.g., Adams \cite{A75}). Fixing the domain of $-\tri$ as $\cD(-\tri):=H^2(\Om)\cap H_0^1(\Om)$, we know that there exists an eigensystem $\{(\la_n,\vp_n)\}_{n=1}^\infty$ of $-\tri$ with the homogeneous Dirichlet boundary condition such that
\[
-\tri\vp_n=\la_n\vp_n,\quad0<\la_1\le\la_2\le\cdots,\quad\la_n\to\infty\mbox{ as }n\to\infty,
\]
and $\{\vp_n\}$ forms a complete orthonormal system of $L^2(\Om)$. Here we number $\la_n$ with their multiplicities. As usual, we can introduce the fractional power $(-\tri)^\ga$ for $\ga\ge0$ as
\[
\cD((-\tri)^\ga):=\left\{h\in L^2(\Om);\sum_{n=1}^\infty|\la_n^\ga(h,\vp_n)|^2<\infty\right\},\quad(-\tri)^\ga h:=\sum_{n=1}^\infty\la_n^\ga(h,\vp_n)\vp_n.
\]
Then $\cD((-\tri)^\ga)$ with $\ga\ge0$ is a Hilbert space equipped with the norm
\[
\|h\|_{\cD((-\tri)^\ga)}:=\left(\sum_{n=1}^\infty|\la_n^\ga(h,\vp_n)|^2\right)^{\f12},\quad h\in\cD((-\tri)^\ga).
\]
Furthermore, there holds $\cD((-\tri)^\ga)\subset H^{2\ga}(\Om)$ for $\ga\ge0$ and especially $\cD((-\tri)^{\f12})=H_0^1(\Om)$. Finally, for a Banach space $X$ and $1\le p\le\infty$, we say that $\Psi\in L^p(0,T;X)$ if
\[
\|\Psi\|_{L^p(0,T;X)}:=\left\{\!\begin{alignedat}{2}
& \left(\int_0^T\|\Psi(\,\cdot\,,t)\|_X^p\right)^{\f1p} & \quad & \mbox{if }1\le p<\infty\\
& \mathop{\mathrm{ess}\sup}_{0<t<T}\|\Psi(\,\cdot\,,t)\|_X & \quad & \mbox{if }p=\infty
\end{alignedat}\right\}<\infty.
\]

Throughout this paper, we assume $f,g\in\cD((-\tri)^{\f{\lceil\al\rceil}2})$ for the source profiles in \eqref{eq-def-F}, i.e., $f\in H_0^1(\Om)$ when $0<\al\le1$ and $f,g\in H^2(\Om)\cap H_0^1(\Om)$ when $1<\al\le2$. For later use, we collect the well-posedness and regularity results of problem \eqref{eq-IBVP-u}--\eqref{eq-def-F} in the following lemma.

\begin{lem}\label{lem-IBVP-u}
Let $f,g\in\cD((-\tri)^{\f{\lceil\al\rceil}2}),$ fix $\ve\in(0,1]$ arbitrarily for $0<\al<2$ and fix $\ve=\f12$ for $\al=2$. Then the following statements hold true.

{\rm(a)} There exists a unique solution $u\in L^\infty(0,T;\cD((-\tri)^{\f{\lceil\al\rceil}2+1-\ve}))$ to \eqref{eq-IBVP-u}--\eqref{eq-def-F} such that $u(\,\cdot\,,t)\longrightarrow0$ in $\cD((-\tri)^{\f{\lceil\al\rceil}2+1-\ve})$ as $t\to0$.

{\rm(b)} If $0<\al\le1,$ then $\pa_t u\in L^1(0,T;\cD((-\tri)^{1-\ve}))$.

{\rm(c)} If $1<\al\le2,$ then
\[
\pa_t u\in L^\infty(0,T;\cD((-\tri)^{2-\f1\al-\ve})),\quad\pa_t^2u\in L^1(0,T;\cD((-\tri)^{\f32-\f1\al-\ve}))
\]
and $\pa_t u(\,\cdot\,,t)\longrightarrow0$ in $\cD((-\tri)^{2-\f1\al-\ve})$ as $t\to0$.
\end{lem}

To conclude the uniqueness for Problem \ref{prob-IMSP}, we need the following vanishing property of the homogeneous problem.

\begin{lem}\label{lem-UCP}
Let $\om\subset\Om$ be an arbitrary nonempty subdomain and $w$ satisfy
\[
\begin{cases}
(\pa_{0+}^\al-\tri)w=0 & \mbox{in }\Om\times(0,T),\\
\begin{cases}
w=a & \mbox{if }0<\al\le1,\\
w=a,\ \pa_t w=b & \mbox{if }1<\al<2
\end{cases} & \mbox{in }\Om\times\{0\},\\
w=0 & \mbox{on }\pa\Om\times(0,T),
\end{cases}
\]
where $a\in L^2(\Om)$ and $b\in\cD((-\tri)^{-\f1\al})$. Then $w=0$ in $\om\times(0,T)$ implies $a=b\equiv0$ in $\Om$.
\end{lem}

For later use, we need the following lemma concerning the long-time asymptotic behavior of the solution to \eqref{eq-IBVP-u} with a source term compactly supported in time.

\begin{lem}\label{lem-asymp}
Let $0<\al<2$ and $u$ be the solution to
\begin{equation}\label{eq-IBVP-u-infty}
\begin{cases}
(\pa_{0+}^\al-\tri)u=F & \mbox{in }\Om\times(0,\infty),\\
u=\pa_t^{\lceil\al\rceil-1}u=0 & \mbox{in }\Om\times\{0\},\\
u=0 & \mbox{on }\pa\Om\times(0,\infty),
\end{cases}
\end{equation}
where $F\in C([0,\infty);C_0^\infty(\ov\Om))$ and there exists $T_0>0$ such that $\supp\,F\subset\ov\Om\times[0,T_0]$. Then
\[
\begin{cases}
J_{0+}^{1-\al}u(\,\cdot\,,T)\longrightarrow0, & 0<\al\le1\\
\pa_{0+}^{\al-1}u(\,\cdot\,,T)\longrightarrow0, & 1<\al<2
\end{cases}\quad\mbox{in }L^2(\Om)\mbox{ as }T\to\infty.
\]
\end{lem}

For the purpose of consistency, we postpone the proofs of the above three lemmas concerning the forward problems to Section \ref{sec-forward}. Below we state the main results in this paper.

\begin{thm}\label{thm-unique}
Let $0<\al\le2,$ $f,g\in\cD((-\tri)^{\f{\lceil\al\rceil}2})$ and $u$ be the solution to \eqref{eq-IBVP-u}--\eqref{eq-def-F}.

{\rm(a)} In the case of $0<\al\le1$, we further assume $\pa\om\supset\pa\Om$ if $\al\ne1$. Then \eqref{eq-u-0} implies $f\equiv0$ in $\Om$.

{\rm(b)} In the case of $1<\al\le2$, we assume $\pa\om\supset\pa\Om$ and, if $\al=2$ we additionally require
\begin{equation}\label{eq-asp-T}
T>2\inf_{\bm y\not\in\ov\Om}\sup_{\bm x\in\Om}|\bm x-\bm y|,\quad B_{c_0 T+\de_0}\subset\Om,
\end{equation}
where $c_0:=\max\{|\bm p|,|\bm q|\}$. Then \eqref{eq-u-0} implies $f=g\equiv0$ in $\Om$.
\end{thm}

\begin{rem}
\begin{itemize}
\item[{\rm(a)}]By examining the proofs of the above theorem in Section \ref{sec-proof}, it turns out that we can consider more general formulations than that in \eqref{eq-IBVP-u}. For instance, instead of $-\tri$ in the governing equation, our argument works for the elliptic operator $-\sum_{j,k=1}^d a_{j k}\pa_j\pa_k+c$, where $(a_{j k})_{1\le j,k\le d}$ is a constant, symmetric and strictly positive definite matrix, and $c\ge0$ is a constant. However, in this paper we choose to treat the simplest model equation in order to focus on the main topic.
\item[{\rm(b)}] In the case $\al=2$,  the two conditions in \eqref{eq-asp-T} can be both satisfied if the sources do not move too fast in comparison with the wave velocity, or equivalently, the maximum velocity $c_0$ is sufficiently small if the wave speed of the background medium has been normalized to be one. Moreover, for $\al=2$, we can prove not only the uniqueness $f=g=0$ but also the stability in estimating $f$ and $g$ by data, but we omit the details. The additional conditions in \eqref{eq-asp-T} for hyperbolic equations are not needed when $0<\al<2$. See e.g. \cite{IY2001} for related inverse problems with $\al=2$.
\end{itemize}
\end{rem}
The determination of a moving source profile function from boundary Cauchy data is stated below. The proof is motivated by recent inverse source problems for acoustic, elastic and electromagnetic wave equations considered in \cite{HKLZ2019, HKZ20, HuKian2020}.

\begin{thm}\label{TH:4}
Let $T=\infty$ and suppose that the temporal function $h$ and the orbit function $\bm\rho$ in \eqref{eq:15} are both given. We assume that
\[
\int_0^\infty h(t)\,\rd t\ne0,\quad\supp\,h\in[0,T_0],\quad f\in C_0^\infty(\Om),\quad\supp\,f\in B_{\de_0}
\]
with some $T_0>0$ and $\de_0>0$.

{\rm(a)} Let $0<\al<2$. We assume that $u$ satisfies \eqref{eq-IBVP-u} and \eqref{eq:15}. Then the source profile $f$ is uniquely determined by Cauchy data $(u,\pa_{\bm\nu}u)$ on $\pa\Om\times(0,\infty)$.

{\rm(b)} Let $\al=2$. We assume that $u$ satisfies {\rm\eqref{eq:15}, \eqref{eq-asp-T}} and the initial value problem
\begin{equation}
\begin{cases}
(\pa_t^2-\tri)u(\bm x,t)=f(\bm x-\bm\rho(t))h(t), & (\bm x,t)\in\Om\times(0,\infty),\\
u(\bm x,0)=\pa_t u(\bm x,0)=0, & \bm x\in\Om.
\end{cases}
\end{equation} 
Then $f$ is uniquely determined by Cauchy data $(u,\pa_{\bm\nu}u)$ on $\pa\Om\times(0,\infty)$.
\end{thm}\bigskip

In Theorems \ref{thm-unique} and \ref{TH:4}, the conclusions for the cases $\al\in(0,1)\cup(1,2)$ and $\al=1,2$ require different assumptions and formulations. This follows from that some properties of solutions in the cases $\al=1$ and $\al=2$ essentially differ from non-integer $\al$. 


\section{Proofs of Lemmas \ref{lem-IBVP-u}--\ref{lem-asymp}}\label{sec-forward}

\begin{proof}[Proof of Lemma $\ref{lem-IBVP-u}$]
First, by the regularity assumption on $f,g$ and the definition \eqref{eq-def-F} of $F$, it follows from the continuity of translation that $F\in\bigcap_{k=0}^{\lceil\al\rceil}C^k([0,T];\cD((-\tri)^{\f{\lceil\al\rceil-k}2}))$. Due to the essential difference in the solution properties, we divide the proofs into the cases of $0<\al<2$ and $\al=2$ separately.\medskip

{\bf Case 1}\ \ For $0<\al<2$, we fix $\ve\in(0,1]$ arbitrarily. In principle, the argument follows the same line as that in Sakamoto and Yamamoto \cite{SY11a} and Li, Liu and Yamamoto \cite{LLY15} especially in the case of $0<\al<1$. For the sake of self-containedness, we still give a proof here.

According to \cite{SY11a}, we can formally write the solution to \eqref{eq-IBVP-u}--\eqref{eq-def-F} as
\begin{equation}\label{eq-rep-u}
u(\,\cdot\,,t)=\int_0^t U(\tau)F(\,\cdot\,,t-\tau)\,\rd\tau,\quad U(t)h:=t^{\al-1}\sum_{n=1}^\infty E_{\al,\al}(-\la_n t^\al)(h,\vp_n)\vp_n.
\end{equation}
For $h\in\cD((-\tri)^\be)$ with some $\be\ge0$ and $0\le\ga<1$, by \eqref{eq-est-ML} we can estimate
\begin{align}
\|U(t)h\|_{\cD((-\tri)^{\be+\ga})}^2 & =t^{2(\al-1)}\sum_{n=1}^\infty|\la_n^\ga E_{\al,\al}(-\la_n t^\al)|^2|\la_n^\be(h,\vp_n)|^2\nonumber\\
& \le(C\,t^{\al-1})^2\sum_{n=1}^\infty\left(\f{(\la_n t^\al)^\ga}{1+\la_n t^\al}t^{-\al\ga}\right)^2|\la_n^\be(h,\vp_n)|^2\nonumber\\
& \le\left(C\|h\|_{\cD((-\tri)^\be)}t^{\al(1-\ga)-1}\right)^2,\quad t>0.\label{eq-est-U}
\end{align}

(a) Taking $\be=\f{\lceil\al\rceil}2$ and $\ga=1-\ve$ in \eqref{eq-est-U}, we employ Minkowski's inequality for integrals to estimate
\begin{align*}
\|u(\,\cdot\,,t)\|_{\cD((-\tri)^{\lceil\al\rceil/2+1-\ve})} & =\left\|\int_0^t U(\tau)F(\,\cdot\,,t-\tau)\,\rd\tau\right\|_{\cD((-\tri)^{\lceil\al\rceil/2+1-\ve})}\\
& \le\int_0^t\|U(\tau)F(\,\cdot\,,t-\tau)\|_{\cD((-\tri)^{\lceil\al\rceil/2+1-\ve})}\,\rd\tau\\
& \le C\int_0^t\|F(\,\cdot\,,t-\tau)\|_{\cD((-\tri)^{\lceil\al\rceil/2})}\tau^{\al\ve-1}\,\rd\tau\\
& \le\f C\ve\|f\|_{\cD((-\tri)^{\lceil\al\rceil/2})}t^{\al\ve},
\end{align*}
which implies (a) immediately.

(b) For $0<\al\le1$, we formally take time derivative in \eqref{eq-rep-u} to deduce
\[
\pa_t u(\,\cdot\,,t)=U(t)F(\,\cdot\,,0)+\int_0^t U(\tau)\pa_t F(\,\cdot\,,t-\tau)\,\rd\tau.
\]
Then taking $\be=0$ and $\ga=1-\ve$ in \eqref{eq-est-U} yields
\begin{align*}
\|\pa_t u(\,\cdot\,,t)\|_{\cD((-\tri)^{1-\ve})} & \le\|U(t)F(\,\cdot\,,0)\|_{\cD((-\tri)^{1-\ve})}+\int_0^t\|U(\tau)\pa_t F(\,\cdot\,,t-\tau)\|_{\cD((-\tri)^{1-\ve})}\,\rd\tau\\
& \le\|F(\,\cdot\,,0)\|_{L^2(\Om)}t^{\al\ve-1}+C\int_0^t\|\pa_t F(\,\cdot\,,t-\tau)\|_{L^2(\Om)}\tau^{\al\ve-1}\,\rd\tau\\
& \le C\|f\|_{L^2(\Om)}t^{\al\ve-1}+\f C\ve\|f\|_{\cD((-\tri)^{1/2})}t^{\al\ve},
\end{align*}
which implies $\pa_t u\in L^1(0,T;\cD((-\tri)^{1-\ve}))$.

(c) For $1<\al<2$, we utilize an alternative expression $u(\,\cdot\,,t)=\int_0^t U(t-\tau)F(\,\cdot\,,\tau)\,\rd\tau$ of \eqref{eq-rep-u}. By $\f\rd{\rd t}(t^{\al-1}E_{\al,\al}(-\la_n t^\al))=t^{\al-2}E_{\al,\al-1}(-\la_n t^\al)$, we formally differentiate the above equality to write
\begin{equation}\label{eq-rep-ut}
\pa_t u(\,\cdot\,,t)=\lim_{\tau\to0}U(\tau)F(\,\cdot\,,t)+\int_0^t V(\tau)F(\,\cdot\,,t-\tau)\,\rd\tau,
\end{equation}
where
\[
V(t)h:=t^{\al-2}\sum_{n=1}^\infty E_{\al,\al-1}(-\la_n t^\al)(h,\vp_n)\vp_n.
\]
For $h\in\cD((-\tri)^\be)$ with some $\be\ge0$, a similar argument as that for \eqref{eq-est-U} yields
\begin{align}
\|V(t)h\|_{\cD((-\tri)^{\be+1-1/\al-\ve})}^2 & =t^{2(\al-2)}\sum_{n=1}^\infty\left|\la_n^{1-\f1\al-\ve}E_{\al,\al-1}(-\la_n t^\al)\right|^2|\la_n^\be(h,\vp_n)|^2\nonumber\\
& \le(C\,t^{\al-2})^2\sum_{n=1}^\infty\left(\f{(\la_n t^\al)^{1-\f1\al-\ve}}{1+\la_n t^\al}t^{-\al(1-\f1\al-\ve)}\right)^2|\la_n^\be(h,\vp_n)|^2\nonumber\\
& \le\left(C\|h\|_{\cD((-\tri)^\be)}t^{\al\ve-1}\right)^2.\label{eq-est-V}
\end{align}
Since $F\in C([0,T];\cD(-\tri))$, we take $\be=1$, $\ga=1-\f1\al-\ve$ in \eqref{eq-est-U} and $\be=1$ in \eqref{eq-est-V} to estimate
\begin{align*}
& \quad\,\|\pa_t u(\,\cdot\,,t)\|_{\cD((-\tri)^{2-1/\al-\ve})}\\
& \le\lim_{\tau\to0}\|U(\tau)F(\,\cdot\,,t)\|_{\cD((-\tri)^{2-1/\al-\ve})}+\int_0^t\|V(\tau)F(\,\cdot\,,t-\tau)\|_{\cD((-\tri)^{2-1/\al-\ve})}\,\rd\tau\\
& \le C\|F(\,\cdot\,,t)\|_{\cD(-\tri)}\lim_{\tau\to0}\tau^{\al\ve}+C\int_0^t\|F(\,\cdot\,,t-\tau)\|_{\cD(-\tri)}\tau^{\al\ve-1}\,\rd\tau\\
& \le\f C\ve\|F\|_{C([0,T];\cD(-\tri))},
\end{align*}
which indicates $\pa_t u\in L^\infty(0,T;\cD((-\tri)^{2-\f1\al-\ve}))$. Finally, within $\cD((-\tri)^{2-\f1\al-\ve})$, we further differentiate \eqref{eq-rep-ut} to deduce
\[
\pa_t^2u(\,\cdot\,,t)=V(t)F(\,\cdot\,,0)+\int_0^t V(\tau)\pa_t F(\,\cdot\,,t-\tau)\,\rd\tau.
\]
Now taking $\be=\f12$ in \eqref{eq-est-V}, we have
\begin{align*}
\|\pa_t^2u(\,\cdot\,,t)\|_{\cD((-\tri)^{3/2-1/\al-\ve})} & \le\|V(t)F(\,\cdot\,,0)\|_{\cD((-\tri)^{3/2-1/\al-\ve})}\\
& \quad\,+\int_0^t\|V(\tau)\pa_t F(\,\cdot\,,t-\tau)\|_{\cD((-\tri)^{3/2-1/\al-\ve})}\,\rd\tau\\
& \le C\|F(\,\cdot\,,0)\|_{\cD((-\tri)^{1/2})}t^{\al\ve-1}+\f C\ve\|\pa_t F\|_{C([0,T];\cD((-\tri)^{1/2}))}t^{\al\ve},
\end{align*}
which completes the proof of (c).\medskip

{\bf Case 2}\ \ For $\al=2$, one can take advantage of the standard theory on hyperbolic equations e.g.\! in \cite{LM72,I06} to conclude $u\in\bigcap_{k=0}^3C^k([0,T];H^{3-k}(\Om))$, which implies the desired results automatically.
\end{proof}

\begin{proof}[Proof of Lemma $\ref{lem-UCP}$]
For $\al=1$, Lemma \ref{lem-UCP} follows from the well-known unique continuation property for parabolic equations (see e.g.\! \cite{SS87}). In the case of $0<\al<1$, Lemma \ref{lem-UCP} reduces to a direct corollary of \cite[Theorem 2.5]{JLLY17}. Hence, in the sequel it suffices to deal with the case of $1<\al<2$.

By Sakamoto and Yamamoto \cite[Theorem 2.3]{SY11a}, we know
\[
w\in C([0,T];L^2(\Om))\cap C((0,T];H^2(\Om)\cap H_0^1(\Om)),
\]
which can be represented as
\[
w(\,\cdot\,,t)=\sum_{n=1}^\infty\{(a,\vp_n)E_{\al,1}(-\la_n t^\al)+(b,\vp_n)\,t\,E_{\al,2}(-\la_n t^\al)\}\vp_n.
\]
Moreover, $w:(0,T]\longrightarrow L^2(\Om)$ can be analytically extended to $(0,\infty)$. Without fear of confusion, we still denote this extension by $w$. Especially, the condition $w=0$ in $\om\times(0,T)$ is also extended to $w=0$ in $\om\times(0,\infty)$. Hence, there holds for any test function $\chi\in C_0^\infty(\om)$ and $t>0$ that
\begin{equation}\label{eq-rep-w}
0=\int_\om w(\,\cdot\,,t)\,\chi\,\rd\bm x=\sum_{n=1}^\infty\{(a,\vp_n)E_{\al,1}(-\la_n t^\al)+(b,\vp_n)\,t\,E_{\al,2}(-\la_n t^\al)\}(\chi,\vp_n),
\end{equation}
where $\chi$ in $(\chi,\vp_n)$ is understood as its zero extension to $\Om$.

Similarly to the proof of \cite[Theorem 4.4]{SY11a}, we attempt to take the Laplace transform of $w$ with respect to $t$. By the estimate (see \cite[Theorem 2.3]{SY11a})
\[
\|w(\,\cdot\,,t)\|_{L^2(\Om)}\le C\left(\|a\|_{L^2(\Om)}+\|b\|_{\cD((-\tri)^{-1/\al})}\right),\quad\forall\,t>0,
\]
we see that for any fixed $z\in\BC$ satisfying $\rRe\,z>0$, the function $\e^{-z t}w(\,\cdot\,,t)$ is integrable with respect to $t\in(0,\infty)$ in $L^2(\Om)$. Employing \eqref{eq-rep-w} and the formula (see Podlubny \cite[\S1.2.2]{P99})
\[
\int_0^\infty\e^{-z t}t^{m-1}E_{\al,m}(-\la_n t^\al)\,\rd t=\f{z^{\al-m}}{z^\al+\la_n},\quad\rRe\,z>\la_1^{1/\al},\ m=1,2,\quad n\in\BN,
\]
we obtain
\begin{equation}\label{eq-Laplace-0}
\sum_{n=1}^\infty\f{(a,\vp_n)z+(b,\vp_n)}{z^\al+\la_n}(\chi,\vp_n)=0,\quad\rRe\,z>\la_1^{1/\al},\ \forall\,\chi\in C_0^\infty(\om).
\end{equation}
Since $z^\al=\exp(\al\log z)$ is not well-defined on the negative real axis, we should cut off this branch and consider $U:=\{z\in\BC;\ -\pi<\arg\,z<\pi\}$. In $U$, we know that the algebraic equation $z^\al+\la_n=0$ with $\la_n>0$ has two distinct roots $z_n^\pm:=\la_n^{1/\al}\exp(\pm\ri\,\f\pi\al)$. Since $a\in L^2(\Om)$ and $b\in\cD((-\tri)^{-\f1\al})$, we can analytically continue both sides of \eqref{eq-Laplace-0} in $z$, so that \eqref{eq-Laplace-0} holds true for $z\in U\setminus\{z_n^\pm\}_{n=1}^\infty$.

To proceed, we shall take into consideration the multiplicity of the eigenvalues of $-\tri$ and rearrange its eigensystem $\{(\la_n,\vp_n)\}$ as follows. By $\{\mu_\ell\}_{\ell=1}^\infty$ we denote the distinct eigenvalues of $-\tri$, and by $\{\psi_{\ell,j}\}_{j=1}^{m_\ell}$ we denote the orthonormal basis of $\ker(\tri+\mu_\ell)$ which coincides with the one in the original eigenfunctions. Then \eqref{eq-Laplace-0} can be rewritten as
\begin{equation}\label{eq-Laplace-1}
\sum_{\ell=1}^\infty\f{(w_\ell(z),\chi)}{z^\al+\mu_\ell}=0,\quad z\in U\setminus\{z_\ell^\pm\}_{\ell=1}^\infty,\ \forall\,\chi\in C_0^\infty(\om),
\end{equation}
where
\[
w_\ell(z):=\sum_{j=1}^{m_\ell}\{(a,\psi_{\ell,j})z+(b,\psi_{\ell,j})\}\psi_{\ell,j}.
\]
Then for any fixed $\ell=1,2,\ldots$ and sufficiently small $\ep>0$, in the $\ep$ neighborhood $B_\ep(z_\ell^\pm)$ of $z_\ell^\pm$ we have
\[
\f{(w_\ell(z),\chi)}{z^\al+\mu_\ell}=-\sum_{k\in\BN\setminus\{\ell\}}\f{(w_k(z),\chi)}{z^\al+\mu_k},\quad z\in B_\ep(z_\ell^\pm)\setminus\{z_\ell^\pm\},\ \forall\,\chi\in C_0^\infty(\om).
\]
Obviously, since
$\sum_{k\in\BN\setminus\{\ell\}}\f{(w_k(z),\chi)}{z^\al+\mu_k}$ is continuous in $z$ at $z_{\ell}^{\pm}$, the right hand side of the above identity is bounded. Multiplying both sides of this identity by $z^\al+\mu_\ell$ and passing $z\to z_\ell^\pm$, we obtain
\[
(w_\ell(z_\ell^\pm),\chi)=\lim_{z\to z_\ell^\pm}(w_\ell(z),\chi)=-\lim_{z\to z_\ell^\pm}(z^\al+\mu_\ell)\sum_{k\in\BN\setminus\{\ell\}}\f{(w_k(z),\chi)}{z^\al+\mu_k}=0,\quad\forall\,\chi\in C_0^\infty(\om).
\]
Then, since $\chi \in C^{\infty}_0(\omega)$ is arbitrary, it follows from the variational principle that
\[
w_\ell(z_\ell^\pm)=\sum_{j=1}^{m_\ell}\left\{(a,\psi_{\ell,j})z_\ell^\pm+(b,\psi_{\ell,j})\right\}\psi_{\ell,j}=0\quad\mbox{in }\om,\ \ell=1,2,\ldots.
\]
Meanwhile, since $w_\ell(z_\ell^\pm)$ satisfy the elliptic equation $(\tri+\mu_\ell)w_\ell(z_\ell^\pm)=0$ in $\Om$, the unique continuation for elliptic equations (e.g., Isakov \cite{I06}) implies $w_\ell(z_\ell^\pm)\equiv0$ in $\Om$ for each $\ell=1,2,\ldots$. By the linear independency of $\{\psi_{\ell,j}\}_{j=1}^{m_\ell}$, we see that
\[
(a,\psi_{\ell,j})z_\ell^\pm+(b,\psi_{\ell,j})=0,\quad1\le j\le m_\ell,\ \ell=1,2,\ldots.
\]
Since $z_\ell^\pm\not\in\BR$ are complex conjugate of each other, we finally obtain
\[
(a,\psi_{\ell,j})=(b,\psi_{\ell,j})=0,\quad1\le j\le m_\ell,\ \ell=1,2,\ldots
\]
and hence $a=b\equiv0$ in $\Om$ due to the completeness of the Dirichlet eigenfunctions.
\end{proof}

\begin{rem}
The proof of Lemma \ref{lem-UCP} for $\al\in(1,2)$ relies heavily on the analyticity of the solution in the time variable, which applies to the scalar wave equation ($\al=2$) when the dynamical measurement data over $(0,\infty)$ are available; see \cite[Theorem 2.1 and Corollary 2.3]{HKZ20} where the data are measured on a closed surface. In Subsection \ref{sec:4.2} below, we shall present a proof for the wave equation using the data over a finite time period $(0,T)$.
\end{rem}

\begin{proof}[Proof of Lemma $\ref{lem-asymp}$]
According to \eqref{eq-rep-u}, we represent the solution to \eqref{eq-IBVP-u-infty} as
\[
u(\,\cdot\,,t)=\sum_{n=1}^\infty u_n(t)\vp_n,\quad u_n(t):=\int_0^t(t-\tau)^{\al-1}E_{\al,\al}(-\la_n(t-\tau)^\al)(F(\,\cdot\,,\tau),\vp_n)\,\rd\tau.
\]

For $0<\al\le1$, we calculate
\begin{align*}
& \quad\,J_{0+}^{1-\al}u_n(T)\\
&  =\f1{\Ga(1-\al)}\int_0^T\f1{(T-t)^\al}\int_0^t(t-\tau)^{\al-1}E_{\al,\al}(-\la_n(t-\tau)^\al)(F(\,\cdot\,,\tau),\vp_n)\,\rd\tau\rd t\\
& =\int_0^T(F(\,\cdot\,,\tau),\vp_n)\left\{\f1{\Ga(1-\al)}\int_\tau^T(T-t)^{-\al}(t-\tau)^{\al-1}E_{\al,\al}(-\la_n(t-\tau)^\al)\,\rd t\right\}\rd\tau,
\end{align*}
where
\begin{align*}
& \quad\,\f1{\Ga(1-\al)}\int_\tau^T(T-t)^{-\al}(t-\tau)^{\al-1}E_{\al,\al}(-\la_n(t-\tau)^\al)\,\rd t\\
& =\f1{\Ga(1-\al)}\sum_{k=0}^\infty\f{(-\la_n)^k}{\Ga(\al(k+1))}\int_\tau^T(T-t)^{-\al}(t-\tau)^{\al(k+1)-1}\rd t\\
& =\f1{\Ga(1-\al)}\sum_{k=0}^\infty\f{(-\la_n)^k}{\Ga(\al(k+1))}(T-\tau)^{\al k}\f{\Ga(1-\al)\Ga(\al(k+1))}{\Ga(\al k+1)}\\
& =\sum_{k=0}^\infty\f{(-\la_n(T-\tau)^\al)^k}{\Ga(\al k+1)}=E_{\al,1}(-\la_n(T-\tau)^\al).
\end{align*}
This implies
\begin{equation}\label{eq-Jwn}
J_{0+}^{1-\al}u_n(T)=\int_0^T E_{\al,1}(-\la_n(T-t)^\al)(F(\,\cdot\,,t),\vp_n)\,\rd t.
\end{equation}

Now we turn to the case of $1<\al<2$. By
$\pa_{0+}^{\al-1}u_n=J_{0+}^{2-\al}(u_n')$, we first calculate $u_n'(t)$. Using
\[
\left(t^{\al-1}E_{\al,\al}(-\la_n\,t^\al)\right)'=\sum_{k=0}^\infty\f{(-\la_n)^k}{\Ga(\al(k+1))}(t^{\al(k+1)-1})'=\sum_{k=0}^\infty\f{(-\la_n)^k t^{\al(k+1)-2}}{\Ga(\al(k+1)-1)},
\]
we have
\[
u_n'(t)=\sum_{k=0}^\infty\f{(-\la_n)^k}{\Ga(\al(k+1)-1)}\int_0^t(t-\tau)^{\al(k+1)-2}(F(\,\cdot\,,\tau),\vp_n)\,\rd\tau.
\]
Then we obtain
\begin{align*}
& \quad\,\pa_{0+}^{\al-1}u_n(T)=J_{0+}^{2-\al}(u_n')(T)\\
& =\f1{\Ga(2-\al)}\int_0^T\f1{(T-t)^{\al-1}}\sum_{k=0}^\infty\f{(-\la_n)^k}{\Ga(\al(k+1)-1)}\int_0^t(t-\tau)^{\al(k+1)-2}(F(\,\cdot\,,\tau),\vp_n)\,\rd\tau\rd t\\
& =\int_0^T(F(\,\cdot\,,\tau),\vp_n)\f1{\Ga(2-\al)}\sum_{k=0}^\infty\f{(-\la_n)^n}{\Ga(\al(k+1)-1)}\int_\tau^T(T-t)^{1-\al}(t-\tau)^{\al(k+1)-2}\rd t\rd\tau\\
& =\int_0^T(F(\,\cdot\,,\tau),\vp_n)\f1{\Ga(2-\al)}\sum_{k=0}^\infty\f{(-\la_n)^n}{\Ga(\al(k+1)-1)}(T-\tau)^{\al k}\f{\Ga(2-\al)\Ga(\al(k+1)-1)}{\Ga(\al k+1)}\\
& =\int_0^T E_{\al,1}(-\la_n(T-t)^\al)(F(\,\cdot\,,t),\vp_n)\,\rd t,
\end{align*}
which takes identically the same form as \eqref{eq-Jwn} in the case of $0<\al\le1$. Then it suffices to investigate
\[
\|J_{0+}^{\lceil\al\rceil-\al}\pa_t^{\lceil\al\rceil-1}u(\,\cdot\,,T)\|_{L^2(\Om)}^2=\sum_{n=1}^\infty\left|\int_0^T E_{\al,1}(-\la_n(T-t)^\al)(F(\,\cdot\,,t),\vp_n)\,\rd t\right|^2.
\]

Since $F(\,\cdot\,,t)=0$ for $t>T_0$, for sufficiently large $T>0$ we choose $\ga>d/4$ arbitrarily to estimate
\begin{align*}
\|J_{0+}^{\lceil\al\rceil-\al}\pa_t^{\lceil\al\rceil-1}u(\,\cdot\,,T)\|_{L^2(\Om)}^2 & =\sum_{n=1}^\infty\left|\int_0^{T_0}E_{\al,1}(-\la_n(T-t)^\al)(F(\,\cdot\,,t),\vp_n)\,\rd t\right|^2\\
& \le\sum_{n=1}^\infty\max_{0\le t\le T_0}|(F(\,\cdot\,,t),\vp_n)|^2\left(\int_{T-T_0}^T|E_{\al,1}(-\la_n\,t^\al)|\,\rd t\right)^2\\
& \le\sum_{n=1}^\infty\f1{\la_n^{2\ga}}\max_{0\le t\le T_0}|\la_n^\ga(F(\,\cdot\,,t),\vp_n)|^2\left(\int_{T-T_0}^T\f{C\,\rd t}{1+\la_n\,t^\al}\right)^2\\
& \le\left(C\,T_0\|F\|_{C([0,T_0];\cD((-\tri)^\ga))}\right)^2\sum_{n=1}^\infty\f1{\la_n^{2\ga}}\f1{(1+\la_n(T-T_0)^\al)^2}\\
& \le C\left(\f{C\,T_0\|F\|_{C([0,T_0];\cD((-\tri)^\ga))}}{1+\la_1(T-T_0)^\al}\right)^2\longrightarrow0\quad(T\to\infty).
\end{align*}
Here we utilized \eqref{eq-est-ML} to estimate $E_{\al,1}(-\la_n\,t^\al)$. Meanwhile, by Courant and Hilbert \cite{CH53}, we know $\la_n\sim n^{2/d}$ and hence $\la_n^{2\ga}\sim n^{4\ga/d}$ with $4\ga/d>1$, so that $\sum_{n=1}^\infty\la_n^{-2\ga}$ converges. The proof of Lemma \ref{lem-asymp} is completed.
\end{proof}

\begin{rem}
For $\al=2$, the solution to the wave equation \eqref{eq-IBVP-u-infty} takes the more explicit form
\[
u(\,\cdot\,,t)=\sum_{n=1}^\infty u_n(t)\vp_n,\quad u_n(t):=\int_0^t\f{\sin(\sqrt{\la_n}(t-\tau))}{\sqrt{\la_n}}(F(\,\cdot\,,\tau),\vp_n)\,\rd\tau.
\]
For $T>T_0$ sufficiently large, we have
\[
\|u(\,\cdot\,,T)\|^2_{L^2(\Om)}=\sum_{n=1}^\infty |u_n(T)|^2=\left|\int_0^{T_0} \f{\sin(\sqrt{\la_n}(T-\tau))}{\sqrt{\la_n}}(F(\,\cdot\,,\tau),\vp_n)\,\rd\tau\right|^2
\]
which does not decay as $T\to\infty$, because the hyperbolic system \eqref{eq-IBVP-u-infty} with the non-absorbing reflecting boundary $\pa\Om$ is not dissipative. Hence, the results of Lemma $\ref{lem-asymp}$ do not carry over to the case of $\al=2$.
\end{rem}

\section{Proof of Theorem \ref{thm-unique}}\label{sec-proof}

This section is devoted to the proof of the first main theorem of this paper concerning the uniqueness for Problem \ref{prob-IMSP}. The key idea originates from the straightforward observation
\begin{equation}\label{eq-vanish}
(\pa_t+\bm p\cdot\nb)f(\bm x-\bm p t)=(\pa_t+\bm p\cdot\nb)(\pa_t+\bm q\cdot\nb)(f(\bm x-\bm p t)+g(\bm x-\bm q t))=0,
\end{equation}
which suggests the introduction of the following auxiliary functions
\begin{equation}\label{eq-auxiliary}
v:=\begin{cases}
J_{0+}^{1-\al}(\pa_t+\bm p\cdot\nb)u, & 0<\al\le1,\\
J_{0+}^{2-\al}(\pa_t+\bm p\cdot\nb)(\pa_t+\bm q\cdot\nb)u, & 1<\al\le2.
\end{cases}
\end{equation}
In such a manner, we can derive a homogeneous equation for $v$ from \eqref{eq-IBVP-u}--\eqref{eq-def-F}, so that Problem \ref{prob-IMSP} is reduced to an inverse problem on determining initial values. To clarify the argument, we deal with the cases of $0<\al\le1$, $1<\al<2$ and $\al=2$ separately.

\subsection{Case of $0<\al\le1$}

First we derive the governing equation for $v_1:=(\pa_t+\bm p\cdot\nb)u$. By Lemma \ref{lem-IBVP-u}, we know $v_1\in L^1(0,T;\cD((-\tri)^{1-\ve}))$ for any $\ve\in(0,1]$. Then by the governing equation in \eqref{eq-IBVP-u} and the definitions of $\pa_{0+}^\al$ and $D_{0+}^\al$, we utilize \eqref{eq-vanish} to formally calculate
\begin{align*}
0 & =(\pa_t+\bm p\cdot\nb)F=(\pa_t+\bm p\cdot\nb)(\pa_{0+}^\al-\tri)u\\
& =(\pa_t J_{0+}^{1-\al})\pa_t u+\pa_{0+}^\al(\bm p\cdot\nb u)-\tri(\pa_t u+\bm p\cdot\nb u)\\
& =D_{0+}^\al(\pa_t+\bm p\cdot\nb)u-\tri(\pa_t+\bm p\cdot\nb)u=(D_{0+}^\al-\tri)v_1,
\end{align*}
where we used Lemma \ref{lem-C-RL}(a) and $\bm p\cdot\nb u=0$ in $\Om\times\{0\}$ to replace $\pa_{0+}^\al(\bm p\cdot\nb u)=D_{0+}^\al(\bm p\cdot\nb u)$. On the other hand, it follows from Lemma \ref{lem-IBVP-u}(a) that we can pass $t\to0$ in \eqref{eq-IBVP-u} to obtain
\[
\lim_{t\to0}J_{0+}^{1-\al}\pa_t u(\,\cdot\,,t)=\lim_{t\to0}\pa_{0+}^\al u(\,\cdot\,,t)=\lim_{t\to0}(\tri u+F)(\,\cdot\,,t)=f\quad\mbox{in }\cD((-\tri)^{\f12-\ve}).
\]
Meanwhile, the weak singularity of $J_{0+}^{1-\al}$ implies $J_{0+}^{1-\al}(\bm p\cdot\nb u)\longrightarrow0$ in $\cD((-\tri)^{1-\ve})$ as $t\to0$. Therefore, we obtain
\[
\lim_{t\to0}J_{0+}^{1-\al}v_1(\,\cdot\,,t)=\lim_{t\to0}J_{0+}^{1-\al}\pa_t u(\,\cdot\,,t)+\lim_{t\to0}J_{0+}^{1-\al}(\bm p\cdot\nb u)(\,\cdot\,,t)=f\quad\mbox{in }\cD((-\tri)^{\f12-\ve}).
\]
Finally, by \eqref{eq-u-0} we have $v_1=0$ in $\om\times(0,T)$. Consequently, it reveals that $v_1$ satisfies an initial-boundary value problem for a time-fractional diffusion equation with the Riemann-Liouville derivative
\[
\begin{cases}
(D_{0+}^\al-\tri)v_1=0 & \mbox{in }\Om\times(0,T),\\
J_{0+}^{1-\al}v_1=f & \mbox{in }\Om\times\{0\},\\
v_1=\bm p\cdot\nb u & \mbox{on }\pa\Om\times(0,T)
\end{cases}
\]
with the additional information $v_1=0$ in $\om\times(0,T)$. If $\al=1$, then $D_{0+}^\al$ reduces to the usual first order derivative $\pa_t$ in time. Then the unique continuation of parabolic equations (e.g. \cite{SS87}) immediately implies $v_1\equiv0$ in $\Om\times(0,T)$ and thus $f\equiv0$ in $\Om$ as the initial value.

In the case of $0<\al<1$, owing to the assumption $\pa\om\supset\pa\Om$ we have $v_1=\bm p\cdot\nb u=0$ on $\pa\Om\times(0,T)$. By further introducing $v:=J_{0+}^{1-\al}v_1$, it is readily seen that
\[
0=J_{0+}^{1-\al}(D_{0+}^\al-\tri)v_1=\pa_{0+}^\al(J_{0+}^{1-\al}v_1)-\tri(J_{0+}^{1-\al}v_1)=(\pa_{0+}^\al-\tri)v\quad\mbox{in }\Om\times(0,T).
\]
In other words, $v$ satisfies
\[
\begin{cases}
(\pa_{0+}^\al-\tri)v=0 & \mbox{in }\Om\times(0,T),\\
v=f & \mbox{in }\Om\times\{0\},\\
v=0 & \mbox{on }\pa\Om\times(0,T).
\end{cases}
\]
Again, we have $v=0$ in $\om\times(0,T)$ by \eqref{eq-u-0}. Then the proof is completed by applying Lemma \ref{lem-UCP}.

\begin{rem}
For $\al=1$, the assumption $\pa\om\supset\pa\Om$ is not necessary because the unique continuation for parabolic equations holds regardless of the boundary condition. However, for $\al\not\in\BN$, the corresponding uniqueness results which are available to our problem are not known. 
\end{rem}


\subsection{Case of $1<\al<2$}

In a parallel manner to the case $0<\al\le 1$, we introduce the auxiliary functions
\[
v_1:=(\pa_t+\bm q\cdot\nb)u,\quad v_2:=(\pa_t+\bm p\cdot\nb)v_1=(\pa_t+\bm p\cdot\nb)(\pa_t+\bm q\cdot\nb)u.
\]
By Lemma \ref{lem-IBVP-u}, we know $v_2\in L^1(0,T;\cD((-\tri)^{\f32-\f1\al-\ve}))$ for any $\ve\in(0,1]$. Similarly as before, we shall derive the governing equation for $v_2$. Since $\pa_t(\bm q\cdot\nb u)=0$ in $\Om\times\{0\}$, we have $\pa_{0+}^\al(\bm q\cdot\nb u)=(\pa_t J_{0+}^{2-\al}\pa_t)(\bm q\cdot\nb u)$ by Lemma \ref{lem-C-RL}(b). Then
\begin{align}
(\pa_t+\bm q\cdot\nb)(\pa_{0+}^\al-\tri)u & =(\pa_t J_{0+}^{2-\al}\pa_t)\pa_t u+(\pa_t J_{0+}^{2-\al}\pa_t)(\bm q\cdot\nb u)
-\tri(\pa_t u+\bm q\cdot\nb u)\nonumber\\
& =(\pa_t J_{0+}^{2-\al}\pa_t-\tri)(\pa_t u+\bm q\cdot\nb u)=(\pa_t J_{0+}^{2-\al}\pa_t-\tri)v_1.\label{eq-govern-v1}
\end{align}
Further, by $v_1=0$ in $\Om\times\{0\}$, Lemma \ref{lem-C-RL}(b) implies $\pa_t J_{0+}^{2-\al}\pa_t v_1=D_{0+}^\al v_1$. Then it follows from \eqref{eq-vanish} and \eqref{eq-govern-v1} that
\begin{align*}
0 & =(\pa_t+\bm p\cdot\nb)(\pa_t+\bm q\cdot\nb)F=(\pa_t+\bm p\cdot\nb)(\pa_t+\bm q\cdot\nb)(\pa_{0+}^\al-\tri)u\\
& =(\pa_t+\bm p\cdot\nb)(\pa_t J_{0+}^{2-\al}\pa_t-\tri)v_1=D_{0+}^\al(\pa_t v_1+\bm p\cdot\nb v_1)-\tri(\pa_t v_1+\bm p\cdot\nb v_1)\\
& =(D_{0+}^\al-\tri)v_2.
\end{align*}

Next, we turn to the initial condition of $v_2$, which involves $J_{0+}^{2-\al}v_2$ and $D_{0+}^{\al-1}v_2$. By the definition of $v_2$ and repeated uses of Lemma \ref{lem-C-RL}(b), we see
\begin{equation}\label{eq-IC-1}
\begin{aligned}
J_{0+}^{2-\al}v_2 & =\pa_{0+}^\al u+J_{0+}^{2-\al}((\bm p+\bm q)\cdot\nb\pa_t u)+J_{0+}^{2-\al}(\bm p\cdot\nb(\bm q\cdot\nb u)),\\
D_{0+}^{\al-1}v_2 & =\pa_t J_{0+}^{2-\al}v_2=\pa_t(\pa_{0+}^\al u)+(\bm p+\bm q)\cdot(\pa_t J_{0+}^{2-\al}\nb\pa_t u)+\pa_t J_{0+}^{2-\al}(\bm p\cdot\nb(\bm q\cdot u))\\
& =\pa_t(\pa_{0+}^\al u)+(\bm p+\bm q)\cdot\nb(\pa_{0+}^\al u)+J_{0+}^{2-\al}(\bm p\cdot\nb(\bm q\cdot\pa_t u)).
\end{aligned}
\end{equation}
Again by Lemma \ref{lem-IBVP-u}, we employ the governing equation of \eqref{eq-IBVP-u} and pass $t\to0$ to find
\begin{equation}\label{eq-IC-2}
\begin{alignedat}{2}
\lim_{t\to0}\pa_{0+}^\al u(\,\cdot\,,t) & =\lim_{t\to0}(\tri u+F)(\,\cdot\,,t)=f+g & \quad & \mbox{in }\cD((-\tri)^{1-\ve}),\\
\lim_{t\to0}\pa_t(\pa_{0+}^\al u)(\,\cdot\,,t) & =\lim_{t\to0}(\tri\pa_t u+\pa_t F)(\,\cdot\,,t)\\
& =-\bm p\cdot\nb f-\bm q\cdot\nb g & \quad & \mbox{in }\cD((-\tri)^{1-\f1\al-\ve}).
\end{alignedat}
\end{equation}
On the other hand, the weak singularity of $J_{0+}^{2-\al}$ and Lemma \ref{lem-IBVP-u} guarantee
\begin{equation}\label{eq-IC-3}
\begin{aligned}
J_{0+}^{2-\al}u(\,\cdot\,,t) & \longrightarrow0\quad\mbox{in }\cD((-\tri)^{2-\ve}),\\
J_{0+}^{2-\al}\pa_t u(\,\cdot\,,t) & \longrightarrow0\quad\mbox{in }\cD((-\tri)^{2-\f1\al-\ve})
\end{aligned}\quad\mbox{as }t\to0.
\end{equation}
Applying \eqref{eq-IC-2} and \eqref{eq-IC-3} and passing $t\to0$ in \eqref{eq-IC-1}, we obtain
\begin{alignat*}{2}
\lim_{t\to0}J_{0+}^{2-\al}v_2(\,\cdot\,,t) & =f+g & \quad & \mbox{in }\cD((-\tri)^{\f32-\f1\al-\ve}),\\
\lim_{t\to0}D_{0+}^{\al-1}v_2(\,\cdot\,,t) & =-\bm p\cdot\nb f-\bm q\cdot\nb g+(\bm p+\bm q)\cdot\nb(f+g)\\
& =\bm q\cdot\nb f+\bm p\cdot\nb g & \quad & \mbox{in }\cD((-\tri)^{1-\f1\al-\ve}).
\end{alignat*}

In conclusion, again it turns out that $v_2$ satisfies the following initial-boundary value problem
\[
\begin{cases}
(D_{0+}^\al-\tri)v_2=0 & \mbox{in }Q,\\
J_{0+}^{2-\al}v_2=f+g,\ D_{0+}^{\al-1}v_2=\bm q\cdot\nb f+\bm p\cdot\nb g & \mbox{in }\Om\times\{0\},\\
v_2=0 & \mbox{on }\pa\Om\times(0,T).
\end{cases}
\]
with the additional information $v_2=0$ in $\om\times(0,T)$ from \eqref{eq-u-0}. Similarly to the case of $0<\al<1$, we further introduce $v:=J_{0+}^{2-\al}v_2$. Then it is readily seen that $v$ satisfies
\begin{equation}\label{eq-IBVP-v}
\begin{cases}
(\pa_{0+}^\al-\tri)v=0 & \mbox{in }\Om\times(0,T),\\
v=f+g,\ \pa_t v=\bm q\cdot\nb f+\bm p\cdot\nb g & \mbox{in }\Om\times\{0\},\\
v=0 & \mbox{on }\pa\Om\times(0,T)
\end{cases}
\end{equation}
with $v=0$ in $\om\times(0,T)$. Taking advantage of Lemma \ref{lem-UCP}, we conclude $f+g=\bm q\cdot\nb f+\bm p\cdot\nb g\equiv0$ in $\Om$. Plugging $g=-f$ in $\bm q\cdot\nb f+\bm p\cdot\nb g=0$ yields $(\bm p-\bm q)\cdot\nb f=0$ in $\Om$, which means that $f$ is a constant along the direction $\bm p-\bm q$. Since we assumed $f\in H_0^1(\Om)$, it should vanish on the boundary, which indicates the vanishing of this constant. In other words, we arrived at $f=g\equiv0$ in $\Om$.

\subsection{Case of $\al=2$}\label{sec:4.2}

Identically parallel to the case of $1<\al<2$, we can introduce $v:=(\pa_t+\bm p\cdot\nb)(\pa_t+\bm q\cdot\nb)u$ and verify that $v$ satisfies (see \eqref{eq-IBVP-v})
\begin{equation}\label{eq-IBVP-v1}
\begin{cases}
(\pa_t^2-\tri)v=0 & \mbox{in }\Om\times(0,T),\\
v=h_0,\ \pa_t v=h_1 & \mbox{in }\Om\times\{0\},\\
v=0 & \mbox{in }\om\times(0,T)
\end{cases}
\end{equation}
with
\[
h_0:=f+g\in H_0^1(\Om),\quad h_1:=\bm q\cdot\nb f+\bm p\cdot\nb g\in L^2(\Om).
\]

Under the conditions in Theorem \ref{thm-unique}(b), it is known that $v=0$ in $\om\times(0,T)$ yields $v(\,\cdot\,,0)=h_0=0$ and $\pa_t v(\,\cdot\,,0)=h_1=0$ in $\Om$ by noting that $v=0$ in $\om\times(0,T)$ yields $v=\pa_{\bm\nu}v=0$ on $\pa\Om\times(0,T)$. We refer to Komornik \cite{Ko} for example, which establishes the stability called an observability inequality implying the desired uniqueness. Thus the proof of Theorem \ref{thm-unique} in the case $\al=2$ is   finished.

As for the uniqueness in determining $h_0$ and $h_1$, we know sharp results, for example, Fritz John's global Holmgren theorem (e.g., Section 1.8 of Chapter 1 in Rauch \cite{Ra}) for a hyperbolic equation with analytic coefficients, but we do not need such sharp uniqueness for our proof.

\section{Proof of Theorem \ref{TH:4} }\label{sec:5}

Due to the linearity of the inverse problem, it suffices to assume $u=\pa_{\bm\nu}u=0$ on $\pa\Om\times(0,\infty)$ and conclude $f=0$ in $\Om$.

For arbitrarily fixed $T>0$, define the test function
\begin{equation}\label{test}
v_T(\bm x,t;\bm\xi):=\e^{-\ri\,\bm\xi\cdot\bm x}(T-t)^{\al-1}E_{\al,\al}(-|\bm\xi|^2(T-t)^\al).
\end{equation}
Then it is readily seen that $\tri v_T(\bm x,t;\bm\xi)=-|\bm\xi|^2v_T(\bm x,t;\bm\xi)$. For $0<\al\le1$, direct calculations yield
\[
J_{T-}^{1-\al}v_T(\bm x,t;\bm\xi)=\e^{-\ri\,\bm\xi\cdot\bm x}E_{\al,1}(-|\bm\xi|^2(T-t)^\al)
\]
and thus
\begin{equation}\label{eq-bRL-vT1}
J_{T-}^{1-\al}v_T(\bm x,T;\bm\xi)=\e^{-\ri\,\bm\xi\cdot\bm x},\quad D_{T-}^\al v_T(\bm x,t;\bm\xi)=|\bm\xi|^2v_T(\bm x,t;\bm\xi).
\end{equation}
For $1<\al\le2$, similar calculations yields
\begin{align*}
& J_{T-}^{2-\al}v_T(\bm x,t;\bm\xi)=\e^{-\ri\,\bm\xi\cdot\bm x}(T-t)E_{\al,2}(-|\bm\xi|^2(T-t)^\al),\\
& D_{T-}^{\al-1}v_T(\bm x,t;\bm\xi)=\e^{-\ri\,\bm\xi\cdot\bm x}E_{\al,1}(-|\bm\xi|^2(T-t)^\al)
\end{align*}
and thus
\begin{equation}\label{eq-bRL-vT2}
\begin{aligned}
& J_{T-}^{2-\al}v_T(\bm x,T;\bm\xi)=0,\quad D_{T-}^{\al-1}v_T(\bm x,T;\bm\xi)=-\e^{-\ri\,\bm\xi\cdot\bm x},\\
& D_{T-}^\al v_T(\bm x,t;\bm\xi)=-|\bm\xi|^2v_T(\bm x,t;\bm\xi).
\end{aligned}
\end{equation}

Based on the test function $v_T$, we investigate the integral
\[
\Phi_T(\bm\xi):=\int_0^T\!\!\!\int_\Om(\pa_{0+}^\al u-\tri u)\,v_T(\,\cdot\,,\,\cdot\,;\bm\xi)\,\rd\bm x\rd t.
\]
By \eqref{eq-IBVP-u} and \eqref{eq:15}, we have
\begin{align*}
\Phi_T(\bm\xi) & =\int_0^T\!\!\!\int_\Om f(\bm x-\bm\rho(t))h(t)\,\e^{-\ri\,\bm\xi\cdot\bm x}(T-t)^{\al-1}E_{\al,\al}(-|\bm\xi|^2(T-t)^\al)\,\rd\bm x\rd t\\
& =\int_0^T\left(\int_\Om f(\bm x-\bm\rho(t))\,\e^{-\ri\,\bm\xi\cdot\bm x}\rd\bm x\right)h(t)(T-t)^{\al-1}E_{\al,\al}(-|\bm\xi|^2(T-t)^\al)\,\rd t.
\end{align*}
Since $f(\bm x-\bm\rho(t))=0$ for all
$\bm x\in\BR^d\setminus\ov\Om$ and $t\in(0,\infty)$, it holds that
\[
\int_\Om f(\bm x-\bm\rho(t))\,\e^{-\ri\,\bm\xi\cdot\bm x}\rd\bm x=\int_{\BR^d}f(\bm x-\bm\rho(t))\,\e^{-\ri\,\bm\xi\cdot\bm x}\rd\bm x=(2\pi)^{d/2}\wt f(\bm\xi)\,\e^{-\ri\,\bm\xi\cdot\bm\rho(t)},
\]
where $\wt f(\bm\xi)=(2\pi)^{-d/2}\int_{\BR^d}f(\bm x)\,\e^{-\ri\,\bm\xi\cdot\bm x}\rd\bm x$ denotes the Fourier transform of $f$. Therefore, we obtain
\begin{equation}\label{eq-PhiT-1}
\Phi_T(\bm\xi)=(2\pi)^{d/2}\wt f(\bm\xi)I_T(\bm\xi),
\end{equation}
where
\begin{equation}\label{eq-def-IT}
I_T(\bm\xi):=\int_0^T\e^{-\ri\,\bm\xi\cdot\bm\rho(t)}h(t)(T-t)^{\al-1}E_{\al,\al}(-|\bm\xi|^2(T-t)^\al)\,\rd t.
\end{equation}

To further treat $\Phi_T(\bm\xi)$, we consider the cases of $0<\al\le1$, $1<\al<2$ and $\al=2$ separately.\medskip

{\bf Case 1 } For $0<\al\le1$ we employ integration by parts and formula \eqref{eq-frac-int1} in Lemma \ref{lem-frac-int} to calculate
\begin{align*}
\Phi_T(\bm\xi) & =\int_\Om\!\int_0^T(\pa_{0+}^\al u)v_T\,\rd t\rd\bm x-\int_0^T\!\!\!\int_\Om(\tri u)v_T\,\rd\bm x\rd t\\
& =\int_\Om\left\{\left[u(J_{T-}^{1-\al}v_T)\right]_0^T-\int_0^T u(D_{T-}^\al v_T)\,\rd t\right\}\rd\bm x\\
& \quad\,+\int_0^T\left\{\int_{\pa\Om}(u\,\pa_{\bm\nu}v_T-v_T\,\pa_{\bm\nu}u)\,\rd\bm\si-\int_\Om u(\tri v_T)\,\rd\bm x\right\}\rd t\\
& =\int_\Om\left(u(J_{T-}^{1-\al}v_T)\right)(\,\cdot\,,T)\,\rd\bm x-\int_0^T\!\!\!\int_\Om u\,(D_{T-}^\al v_T+\tri v_T)\,\rd\bm x\rd t\\
& =\int_\Om u(\bm x,T)\,\e^{-\ri\,\bm\xi\cdot\bm x}\rd\bm x.
\end{align*}
Here we used the facts that $u=0$ in $\Om\times\{0\}$, $u=\pa_{\bm\nu}u=0$ on $\pa\Om\times(0,\infty)$ and $(D_{T-}^\al+\tri)v_T=0$ due to \eqref{eq-bRL-vT1}. Combining the above equality with \eqref{eq-PhiT-1} indicates
\begin{equation}\label{eq-PhiT-2}
(2\pi)^{d/2}\wt f(\bm\xi)I_T(\bm\xi)=\int_\Om u(\bm x,T)\,\e^{-\ri\,\bm\xi\cdot\bm x}\rd\bm x
\end{equation}
for all $T>0$ and $\bm\xi\in\BR^d$. Further taking $J_{0+}^{1-\al}$ on both sides of \eqref{eq-PhiT-2} and passing $T\to\infty$, we obtain
\begin{equation}\label{eq-PhiT-3}
(2\pi)^{d/2}\wt f(\bm\xi)\lim_{T\to\infty}J_{0+}^{1-\al}I_T(\bm\xi)=\lim_{T\to\infty}\int_\Om J_{0+}^{1-\al}u(\bm x,T)\,\e^{-\ri\,\bm\xi\cdot\bm x}\rd\bm x,\quad\bm\xi\in\BR^d.
\end{equation}
For the right-hand side of \eqref{eq-PhiT-3}, it follows from Lemma \ref{lem-asymp} that
\begin{align}
\left|\int_\Om J_{0+}^{1-\al}u(\bm x,T)\,\e^{-\ri\,\bm\xi\cdot\bm x}\rd\bm x\right| & \le\|J_{0+}^{1-\al}u(\,\cdot\,,T)\|_{L^1(\Om)}\nonumber\\
& \le\sqrt{|\Om|}\,\|J_{0+}^{1-\al}u(\,\cdot\,,T)\|_{L^2(\Om)}\longrightarrow0\quad(T\to\infty).\label{eq-PhiT-4}
\end{align}

Now we investigate $\lim_{T\to\infty}J_{0+}^{1-\al}I_T(\bm\xi)$ where $I_T(\bm\xi)$ was defined in \eqref{eq-def-IT}. Taking $\bm\xi=\bm0$ in $I_T(\bm\xi)$ gives
\[
J_{0+}^{1-\al}I_T(\bm0)=J_{0+}^{1-\al}\left\{\f1{\Ga(\al)}\int_0^T h(t)(T-t)^{\al-1}\rd t\right\}=J_{0+}^{1-\al}(J_{0+}^\al h)(T)=\int_0^T h(t)\,\rd t.
\]
Then it follows immediately from the assumption of $h$ that
\[
\lim_{T\to\infty}J_{0+}^{1-\al}I_T(\bm0)=\int_0^\infty h(t)\,\rd t\ne0.
\]
Now that $\lim_{T\to\infty}J_{0+}^{1-\al}I_T(\bm\xi)$ is independent of $T$ and is continuous with respect to $\bm\xi$, there exists a sufficiently small constant $\ep>0$ such that
\[
\lim_{T\to\infty}J_{0+}^{1-\al}I_T(\bm\xi)\ne0,\quad\forall\,\bm\xi\in B_\ep.
\]
In view of \eqref{eq-PhiT-3} and \eqref{eq-PhiT-4}, we obtain $\wt f=0$ in $B_\ep$. Finally, the analyticity of $\wt f$ in $\bm\xi\in\BR^d$ implies $\wt f=0$ in $\BR^d$ and consequently $f=0$ in $\Om$.\medskip

{\bf Case 2 } Now we turn to the case of $1<\al<2$. In a similar manner as before, we employ formula \eqref{eq-frac-int2} in Lemma \ref{lem-frac-int} to calculate
\begin{align*}
\Phi_T(\bm\xi) & =\int_\Om\left\{\left[(\pa_t u)J_{T-}^{2-\al}v_T\right]_0^T-\int_0^T(\pa_t u)D_{T-}^\al v_T\,\rd t\right\}\rd\bm x\\
& \quad\,+\int_0^T\left\{\int_{\pa\Om}(u\,\pa_{\bm\nu}v_T-v_T\,\pa_{\bm\nu}u)\,\rd\bm\si-\int_\Om u(\tri v_T)\,\rd\bm x\right\}\rd t\\
& =\int_\Om\left\{\left[u(D_{T-}^{\al-1}v_T)\right]_T^0+\int_0^T u(D_{T-}^\al v_T)\,\rd t\right\}-\int_0^T\!\!\!\int_\Om u(\tri v_T)\,\rd\bm x\rd t\\
& =-\int_\Om\left(u(D_{T-}^{\al-1}v_T)\right)(\,\cdot\,,T)\,\rd\bm x+\int_0^T\!\!\!\int_\Om u\,(D_{T-}^\al v_T-\tri v_T)\,\rd\bm x\rd t\\
& =\int_\Om u(\bm x,T)\,\e^{-\ri\,\bm\xi\cdot\bm x}\rd\bm x.
\end{align*}
Here we used the facts that $u=\pa_t u=0$ in $\Om\times\{0\}$, $u=\pa_{\bm\nu}u=0$ on $\pa\Om\times(0,\infty)$ and $(D_{T-}^\al-\tri)v_T=0$ due to \eqref{eq-bRL-vT2}. This means that we arrive at the same equality \eqref{eq-PhiT-2} as that in Case 1. Nevertheless, this time we take $\pa_{0+}^{\al-1}$ on both sides of \eqref{eq-PhiT-2} and pass $T\to\infty$ to obtain
\[
(2\pi)^{d/2}\wt f(\bm\xi)\lim_{T\to\infty}\pa_{0+}^{\al-1}I_T(\bm\xi)=\lim_{T\to\infty}\int_\Om\pa_{0+}^{\al-1}u(\bm x,T)\,\e^{-\ri\,\bm\xi\cdot\bm x}\rd\bm x,\quad\bm\xi\in\BR^d.
\]
Again the right-hand side vanishes due to Lemma \ref{lem-asymp}. Similarly as before, we take $\bm\xi=\bm0$ in $\pa_{0+}^{\al-1}I_T(\bm\xi)$ to see
\begin{align*}
\pa_{0+}^{\al-1}I_T(\bm0) & =J_{0+}^{2-\al}\left\{\f1{\Ga(\al)}\f\rd{\rd T}\int_0^T h(t)(T-t)^{\al-1}\rd t\right\}\\
& =J_{0+}^{2-\al}\left\{\f1{\Ga(\al-1)}\int_0^T h(t)(T-t)^{\al-2}\rd t\right\}=J_{0+}^{2-\al}(J_{0+}^{\al-1}h)(T)=\int_0^T h(t)\,\rd t.
\end{align*}
By the same continuity argument, we can conclude $\wt f=0$ in a neighborhood of $\bm0$ and eventually complete the proof.\medskip

{\bf Case 3 } Consider the last case $\al=2$. We recall that $\supp\,f\subset\Om$. In this case, the test function \eqref{test} takes the form
\[
v_T(\bm x,t;\bm\xi)=\e^{-\ri\bm\xi\cdot\bm x}\f{\sin(|\bm\xi|(T-t))}{|\bm\xi|},
\]
which is not well-defined if $|\bm\xi|=0$. For wave equations, we prefer to use another test function defined by
\[
v(\bm x,t;\bm\xi)=\e^{-\ri|\bm\xi|t}\e^{-\ri\bm\xi\cdot\bm x},\quad(\bm x,t)\in\BR^d\times(0,\infty),\quad\bm\xi\in\BR^d.
\]
which satisfies the wave equation $(\pa_t^2-\tri)v(\bm x,t;\bm\xi)=0$ for $\bm x,\bm\xi\in\BR^d$. As done in Cases 1 and 2, it follows from \eqref{eq:15} that
\begin{equation}\label{eq:16}
\int_0^T h(t)\int_\Om f(\bm x-\bm\rho(t))\,v(\bm x,t;\bm\xi)\,\rd\bm x\rd t=\int_0^T\!\!\!\int_\Om(\pa_t^2-\tri)u(\bm x,t)\,v(\bm x,t;\bm\xi)\,\rd\bm x\rd t.
\end{equation}
Inserting the expression of $v$ to the left hand side, we obtain
\[
\int_0^T h(t)\int_\Om f(\bm x-\bm\rho(t))\,v(\bm x,t;\bm\xi)\,\rd\bm x\rd t=(2\pi)^{d/2}\wt f(\bm\xi)\int_0^T\e^{-\ri\bm\xi\cdot\bm\rho(t)}\,h(t)\,\rd t,
\]
where $\wt f$ again denotes the Fourier transform of the source function $f$. For the right hand side, using integration by parts together with the vanishing of the Cauchy data $(u,\pa_{\bm\nu}u)$ on $\pa\Om\times(0,\infty)$, we obtain
\begin{align*}
\int_0^T\!\!\!\int_\Om\pa_t^2u(\bm x,t)\,v(\bm x,t;\bm\xi)\,\rd\bm x\rd t & =\int_\Om((\pa_t u)v-u\,\pa_t v)|_{t=T}\,\rd\bm x\\
& \quad\,+\int_0^T\!\!\!\int_\Om\pa_t^2v(\bm x,t;\bm\xi)\,u(\bm x,t)\,\rd\bm x\rd t,\\
\int_0^T\!\!\!\int_\Om\tri u(\bm x,t)\,v(\bm x,t;\bm\xi)\,\rd\bm x\rd t & =\int_0^T\!\!\!\int_\Om u(\bm x,t)\,\tri v(\bm x,t;\bm\xi)\,\rd\bm x\rd t,
\end{align*}
implying that
\[
\int_0^T\!\!\!\int_\Om(\pa_t^2-\tri )u(\bm x,t)\,v(\bm x,t;\bm\xi)\,\rd\bm x\rd t=\int_\Om((\pa_t u)v-u\,\partial_t v)|_{t=T}\,\rd\bm x.
\]
Hence, the identity \eqref{eq:16} can be rewritten as
\[
(2\pi)^{d/2}\wt f(\bm\xi)\int_0^T\e^{-\ri\bm\xi\cdot\bm\rho(t)}h(t)\,\rd t=\int_\Om((\pa_t u)v-u\,\pa_t v)|_{t=T}\,\rd\bm x.
\]
Since $u=\pa_{\bm \nu}u=0$ on $\pa\Om\times(0,\infty)$, one may extend $u(\bm x,t)$ from $\Om\times\BR_+$ to $(\BR^d\setminus\ov\Om)\times\BR_+$ by zero. Hence, $u$ can be given explicitly by the convolution of the fundamental solution of the wave equation and the source term. This implies that the right hand side of the previous identity tends to zero as $T\to\infty$, because $\|u(\,\cdot\,,T)\|_{L^2(\Om)}\longrightarrow0$ and $\|\pa_t u(\,\cdot\,,T)\|_{L^2(\Om)}\longrightarrow0$ as $T\to\infty$. Since $h$ is compactly supported in $[0,T_0]$, letting $T\rightarrow\infty$ we arrive at
\[
(2\pi)^{d/2}\wt f(\bm\xi)\int_0^{T_0}\e^{-\ri\bm\xi\cdot\bm\rho(t)}h(t)\,\rd t=0
\]
for all $|\bm\xi|<\ep$. By the assumption that $\int_0^\infty h(t)\,\rd t\ne0$, we thus obtain $\wt f(\bm\xi)=0$, provided $|\bm \xi|$ is sufficiently small. By the analyticity of the function $\bm\xi\longmapsto\wt f(\bm\xi)$, we have $\wt f\equiv0$ in $\BR^d$ and thus $f\equiv0$ by taking the inverse Fourier transform.

\begin{rem}
In Theorem \ref{TH:4}, the non-vanishing assumption $\int_0^\infty h(t)\,\rd t\ne0$ can be replaced by other conditions. For example, in the case $\al=2$ (wave equation), one can apply the moment theory to prove the unique determination of the orbit $\bm\rho(t)\in\BR^3$ from boundary Cauchy data, provided the starting position of the moving source is known and the source moves more slowly than the wave speed. A detailed argument in electromagnetism can be found in \cite[Theorem 4.2]{HKLZ2019}. We conjecture that the moment theory also applies to the fractional diffusion-wave equation under additional assumptions on the orbit function $\bm\rho(t)$ and the temporal function $h(t)$.
\end{rem}


\section{Concluding remarks}\label{sec-conclusion}

In this paper, we mainly investigated an inverse moving source problem on determining source profiles in (time-fractional) evolution equations, provided that the sources move along given constant vectors. Under some assumptions on the observation subdomain $\om$ and the observation time $T$, we proved the unique determination of at most $\lceil\al\rceil$ unknown profiles, where $\al\in(0,2]$ is fractional derivative order. The key to the proof turns out to be a reduction to an inverse problem for initial conditions by introducing auxiliary functions \eqref{eq-auxiliary}. Then for the homogeneous problems, we employ a vanishing property for $0<\al<2$ and a Carleman estimate argument for $\al=2$ to conclude the uniqueness. Using boundary Cauchy data over infinite time, unique determination of source profiles with a general orbit and the track of a moving point source along a hyperplane in 3D is also discussed.

We close this paper by mentioning several possible future topics on inverse moving source problems. In this work, the unknown sources are assumed to move along straight lines, which seems unrealistic in most situations. Therefore, it is preferable to remove this assumption and consider general given orbits (for instance, sources moving along a plane in $\BR^3$). Similarly, the observation subdomain $\om$ is assumed to cover the whole boundary, which also looks restrictive. We shall attempt to relax this condition by seeking new vanishing property which does not require homogeneous boundary conditions. Meanwhile, another related issue is to study the same problem by using partial boundary Cauchy data. Finally, in the light of practical applications, it is necessary to develop corresponding numerical methods and perform numerical verifications.
\bigskip

{\bf Acknowledgement}\ \ The authors thank the anonymous referees for their valuable comments. This paper has been supported by Grant-in-Aid for Scientific Research (S) 15H05740, Japan Society for the Promotion of Science (JSPS) and by the RUDN University Strategic Academic Leadership Program. Y.\! Liu is supported by Grant-in-Aid for Early-Career Scientists 20K14355, JSPS. G.\! Hu is partly supported by National Natural Science Foundation of China (NSFC) (No.\! 12071236) and Fundamental Research Funds for the Central Universities in China. M.\! Yamamoto is supported by Grant-in-Aid for Scientific Research (A) 20H00117, JSPS and by NSFC (Nos.\! 11771270, 91730303).



\begin{thebibliography}{99}

\bibitem{AG92}
Adams E E and Gelhar L W 1992 Field study of dispersion in an heterogeneous aquifer 2. Spatial moments analysis {\it Water Res. Res.} {\bf28} 3293--307

\bibitem{A75}
Adams R A 1975 {\it Sobolev Spaces} (New York: Academic)

\bibitem{BWM00}
Benson D A, Wheatcraft S W and Meerschaert M M 2000 Application of a fractional advection-dispersion equation {\it Water Res. Res.} {\bf36} 1403--12

\bibitem{BDES18}
Brown T S, Du S, Ersulu H and Sayas F-J 2018 Analysis of models for viscoelastic wave propagation {\it Appl. Math. Nonlinear Sci.} {\bf3} 55--96

\bibitem{CH53}
Courant R and Hilbert D 1953 {\it Methods of Mathematical Physics} vol 1 (New York: Interscience)

\bibitem{EK04}
Eidelman S D and Kochubei A N 2004 Cauchy problem for fractional diffusion equations {\it J. Differ. Equ.} {\bf199} 211--55

\bibitem{FK16}
Fujishiro K and Kian Y 2016 Determination of time dependent factors of coefficients in fractional diffusion equations {\it Math. Control Relat. Fields} {\bf6} 251--69

\bibitem{GCR92}
Ginoa M, Cerbelli S and Roman H E 1992 Fractional diffusion equation and relaxation in complex viscoelastic materials {\it Phys. A} {\bf191} 449--53

\bibitem{GLY15}
Gorenflo R, Luchko Y and Yamamoto M 2015 Time-fractional diffusion equation in the fractional Sobolev spaces {\it Frac. Calc. Appl. Anal.} {\bf18} 799--820

\bibitem{HH98}
Hatano Y and Hatano N 1998 Dispersive transport of ions in column experiments: an explanation of long-tailed profiles {\it Water Res. Res.} {\bf34} 1027--33

\bibitem{HuKian2020}
Hu G and Kian Y 2020 Uniqueness and stability for the recovery of a time dependent source in elastodynamics {\it Inverse Probl. Imaging} {\bf14} 463--87

\bibitem{HKLZ2019}
Hu G, Kian Y, Li P and Zhao Y 2019 Inverse moving source problems in electrodynamics {\it Inverse Problems} {\bf35} 075001

\bibitem{HKZ20}
Hu G, Kian Y and Zhao Y 2020 Uniqueness to some inverse source problems for the wave equation in unbounded domains {\it Acta Math. Appl. Sin. Engl. Ser.} {\bf36} 134--50

\bibitem{HLY20}
Hu G, Liu Y and Yamamoto M 2020 Inverse moving source problem for fractional diffusion(-wave) equations: Determination of orbits {\it Inverse Problems and Related Topics} ed J Cheng, S Lu and M Yamamoto (Singapore: Springer) pp 81--100

\bibitem{IY2001}
Imanuvilov O Y and Yamamoto M 2001 Global Lipschitz stability in an inverse hyperbolic problem by interior observation {\it Inverse Problems} {\bf17} 717--28

\bibitem{I06}
Isakov V 2006 {\it Inverse Problems for Partial Differential Equations} (New York: Springer)

\bibitem{JLLY17}
Jiang D, Li Z, Liu Y and Yamamoto M 2017 Weak unique continuation property and a related inverse source problem for time-fractional diffusion advection equations {\it Inverse Problems} {\bf33} 055013

\bibitem{Ko}
Komornik V 1994 {\it Exact Controllability and Stabilization: the Multiplier Method} (Chichester: Wiley)

\bibitem{KY18}
Kubica A and Yamamoto M 2018 Initial-boundary value problems for fractional diffusion equations with time-dependent coefficients {\it Fract. Calc. Appl. Anal.} {\bf21} 276--311

\bibitem{LB03}
Levy M and Berkowitz B 2003 Measurement and analysis of non-Fickian dispersion in heterogeneous porous media {\it J. Contaminant Hydrology} {\bf64} 203--26

\bibitem{LLY15}
Li Z, Liu Y and Yamamoto M 2015 Initial-boundary value problems for multi-term time-fractional diffusion equations with positive constant coefficients {\it Appl. Math. Comput.} {\bf257} 381--97

\bibitem{LiLiuY19}
Li Z, Liu Y and Yamamoto M 2019 Inverse problems of determining parameters of the fractional partial differential equations {\it Handbook of Fractional Calculus with Applications} vol 2 ed J A Tenreiro Machado, A N Kochubei and Y Luchko (Berlin: De Gruyter) pp 431--42

\bibitem{LiY19}
Li Z and Yamamoto M 2019 Inverse problems of determining coefficients of the fractional partial differential equations {\it Handbook of Fractional Calculus with Applications} vol 2 ed J A Tenreiro Machado, A N Kochubei and Y Luchko (Berlin: De Gruyter) pp 443--64

\bibitem{LM72}
Lions J-L and Magenes E 1972 {\it Non-Homogeneous Boundary Value Problems and Applications} (Berlin: Springer)

\bibitem{L21}
Liu Y 2021 Numerical schemes for reconstructing profiles of moving sources in (time-fractional) evolution equations {\it RIMS K\^oky\^uroku} {\bf2174} 73--87

\bibitem{LiuLiY19}
Liu Y, Li Z and Yamamoto M 2019 Inverse problems of determining sources of the fractional partial differential equations {\it Handbook of Fractional Calculus with Applications} vol 2 ed J A Tenreiro Machado, A N Kochubei and Y Luchko (Berlin: De Gruyter) pp 411--29

\bibitem{LRY16}
Liu Y, Rundell W and Yamamoto M 2016 Strong maximum principle for fractional diffusion equations and an application to an inverse source problem {\it Frac. Calc. Appl. Anal.} {\bf19} 888--906

\bibitem{LZ17}
Liu Y and Zhang Z 2017 Reconstruction of the temporal component in the source term of a (time-fractional) diffusion equation {\it J Phys A.} {\bf50} 305203

\bibitem{L09}
Luchko Y 2009 Maximum principle for the generalized time-fractional diffusion equation {\it J. Math. Anal. Appl.} {\bf351} 218--23

\bibitem{NIO12}
Nakaguchi E, Inui H and Ohnaka K 2012 An algebraic reconstruction of a moving point source for a scalar wave equation {\it Inverse Problems} {\bf28} 065018

\bibitem{O20}
Ohe T 2020 Real-time reconstruction of moving point/dipole wave sources from boundary measurements {\it Inverse Probl. Sci. Eng.} {\bf28} 1057--102

\bibitem{P99}
Podlubny I 1999 {\it Fractional Differential Equations} (San Diego, CA: Academic)

\bibitem{Ra}
Rauch J 1991 {\it Partial Differential Equations} (Berlin: Springer)

\bibitem{SY11a}
Sakamoto K and Yamamoto M 2011 Initial value/boundary value problems for fractional diffusion-wave equations and applications to some inverse problems {\it J. Math. Anal. Appl.} {\bf382} 426--47

\bibitem{SY11b}
Sakamoto K and Yamamoto M 2011 Inverse source problem with a final overdetermination for a fractional diffusion equation {\it Math. Control Relat. Fields} {\bf1} 509--18

\bibitem{SS87}
Saut J C and Scheurer B 1987 Unique continuation for some evolution equations {\it J. Differ. Equ.} {\bf66} 118--39

\end{thebibliography}
\end{document}